\documentclass{amsart}
\usepackage{amssymb}

\newtheorem{Theorem}{Theorem}[section]

\newtheorem{Lemma}[Theorem]{Lemma}
\newtheorem{Proposition}[Theorem]{Proposition}
\newtheorem{Remark}{Remark}[section]

\title[Lifespan for 1D Damped Wave eq. with 0 0-th Fourier moment]{Lifespan estimates for 1d damped wave equation with zero moment initial data}
\author[K. Fujiwara]{Kazumasa Fujiwara}

\address[K. Fujiwara]{
Faculty of Advanced Science and Technology,
Ryukoku University,
1-5 Yokotani,Seta Oe-cho,Otsu,Shiga, Japan
}

\email{fujiwara.kazumasa@math.ryukoku.ac.jp}

\author[V. Georgiev]{Vladimir Georgiev}

\address[V.Georgiev]{
Department of Mathematics,
University of Pisa,
Largo Bruno Pontecorvo 5,
I - 56127 Pisa, Italy}
\address{
Faculty of Science and Engineering, Waseda University,
3-4-1, Okubo, Shinjuku-ku, Tokyo 169-8555, Japan}
\address{
Institute of Mathematics and Informatics,  Bulgarian Academy of Sciences, Acad. Georgi Bonchev Str., Block 8, Sofia, 1113, Bulgaria
}
\email{georgiev@dm.unipi.it}


\begin{document}

\begin{abstract}
In this manuscript,
a sharp lifespan estimate
of solutions to semilinear classical damped wave equation
is investigated
in one dimensional case
when the Fourier 0th moment of sum of initial position and speed is $0$.
Especially,
it is shown that the behavior of lifespan
changes with $p=3/2$ with respect to the size of the initial data.
\end{abstract}

\maketitle

\section{Introduction}
In this manuscript,
we study the Cauchy problem of the following classical 1d damped wave equation:
    \begin{align}
    \begin{cases}
    \partial_t^2 u + \partial_t u - \Delta u = |u|^{p},
    &\ t \in (0,T_0), \quad x \in \mathbb R,\\
    u(0) = u_0,
    &\ x \in  \mathbb{R},\\
    \partial_t u(0) = u_1,
    &\ x \in \mathbb R.
    \end{cases}
    \label{eq:DW}
    \end{align}
Our goal is to obtain optimal lifespan estimates of solutions to \eqref{eq:DW}
when the $0$th moment of sum of initial position and speed is $0$.

Extensive studies have been conducted
on the existence and nonexistence of time-global solutions
to \eqref{eq:DW} for small initial data
in the framework of perturbation argument.
Namely, 
the maximal existence time has been investigated
through the examination of quantities
that reflect resemblance between
the behavior of solutions to \eqref{eq:DW}
and that of their corresponding free linear solutions.
We note that
the Fujita exponent $p_F(1)=3$ is known to be the critical exponent to \eqref{eq:DW},
that is, criteria of the existence and nonexistence of time-global solutions
for small data in the sense explained below.

In the case where $p > 3$,
the global existence of mild solutions is well-known from the article of
Todorova and Yordanov \cite{TY01}
if $(u_0,u_1) \in W^{1,2} \times L^2$ is sufficiently small and compactly supported.
Here $W^{1,2}$ denotes the usual Sobolev space of 1st order based on $L^2$
defined by the collection of all $L^2$ functions $f$ satisyfing $f, f' \in L^2$.
We also call $u \in C([0,\infty); W^{1,2}) \cap C^1((0,\infty); L^2)$
as a mild solution to \eqref{eq:DW} if $u$ satisfies the integral form of \eqref{eq:DW}:
    \begin{align}
    u(t) = S(t)(u_0+u_1) + \partial_t S(t) u_0    
    + \int_0^t S(t-\tau) |u(\tau)|^p d \tau,
    \label{eq:IDW}
    \end{align}
where $S(t)$ is given by
    \[
    S(t) f(x)
    = \frac 1 2 \int_{-t}^t I_0 \bigg( \frac{\sqrt{t^2-y^2}}{2} \bigg) f(x-y) dy
    \]
and $I_0$ is the modified Bessel function of $0$ order whose Taylor series is 
    \[
    I_0 (y) = \sum_{k \geq 0} \frac{1}{(k!)^2} \bigg( \frac{y}{2} \bigg)^2.
    \tag*{\cite[8.445]{GR07}}
    \]
We note that the analysis of \cite{TY01} is based on a modified energy estimate.
Moreover, Marcati and Nishara \cite{MN03} showed the global existence
for small $W^{1,2} \times L^2$ data
without compactness of support of initial data.
Especially, their argument is based on the following heat-like $L^p-L^q$ estimate
for fundamental solutions to \eqref{eq:DW}:
\begin{Lemma}[{\cite[Theorem 1.2]{MN03}}]
\label{Lemma:SEstimate}
For $1 \leq q \leq p \leq \infty$ and positive $\alpha, \beta$ with $1 \leq \alpha + \beta \leq 3$,
the following estimates hold:
    \begin{align*}
    \| S(t) f \|_{L^p}
    &\leq C (1+t)^{- \frac 1 2 (\frac 1 q - \frac 1 p)} \| f \|_{L^q},\\
    \bigg\| \partial_x^\alpha \partial_t^\beta \big( S(t) f - e^{-t/2} W(t) f \big) \bigg\|_{L^p}
    &\leq C (1+t)^{- \frac 1 2 (\frac 1 q - \frac 1 p) - \frac \alpha 2 - \beta} \| f \|_{L^q},
    \end{align*}
where $W(t)$ is the 1-D wave operator defined by
    \[
    W(t) f(x) = \frac 1 2 \int_{x-t}^{x+t} f(y) dy.
    \]
Moreover, the estimates
    \begin{align*}
    \| S(t) f - e^{t \Delta} f \|_{L^p}
    &\leq C (1+t)^{- \frac 1 2 (\frac 1 q - \frac 1 p) - 1} \| f \|_{L^q},\\
    \bigg\| \partial_x^\alpha \partial_t^\beta \big( S(t) f - e^{t \Delta} f
    - e^{-t/2} W(t) f \big) \bigg\|_{L^p}
    &\leq C (1+t)^{- \frac 1 2 (\frac 1 p - \frac 1 q) - \frac \alpha 2 - \beta - 1} \| f \|_{L^q}
    \end{align*}
hold, where $e^{t \Delta}$ is the 1-D heat operator defined by
    \[
    e^{t \Delta} f(x) = \frac{1}{\sqrt{2 \pi t}} \int_{-\infty}^\infty e^{-\frac{|x-y|^2}{4t}} f(y) dy.
    \]
\end{Lemma}
\noindent
Applying Lemma \ref{Lemma:SEstimate} and employing the standard contraction argument to \eqref{eq:IDW},
the small data global existence of mild solutions is established.
We note that the three dimensional case is treated in \cite{N03} as well.
Later, Ikeda, Inui, Okamoto, and Wakasughi extended this approach by Fourier analysis of $S(t)$ in \cite{IIMW19}.
We also refer \cite{M97} as a pioneering work.

In the case where $p \leq 3$,
the global non-existence and corresponding lifespan estimate have been studied
when
        \begin{align}
    M_0 (u_0+u_1) >0,
    \label{eq:M0positive}       
    \end{align}
where $M_0$ is the Fourier $0$-th moment defined by
        \[
    M_0(f) = \int_{\mathbb R^n} f(x) dx
    \]
for $f \in L^1$.
Indeed,
let
$p \leq 3$ and
$f_0,f_1 \in L^1 \cap L^p$ is sufficiently regular and satisfies \eqref{eq:M0positive},
then
there exist positive constants $c, C$ depends on $(f_0,f_1)$
such that for sufficiently small $\varepsilon > 0$,
the lifespan $T_0 = T_0(u_0,u_1)$ is estimated by
    \begin{align}
    T_p(c \varepsilon)
    \leq T_0
    \leq T_p(C \varepsilon)
    \label{eq:T0TpEstimate}
    \end{align}
with $(u_0,u_1) = (\varepsilon f_0, \varepsilon f_1)$,
where
    \[
    T_{p}(\varepsilon)
    = \begin{cases}
    \varepsilon^{- \frac{2(p-1)}{3-p}}
    & \mathrm{if} \quad p \in (1, 3),\\
    \exp(\varepsilon^{- 2 }) & \mathrm{if} \quad p = 3.
    \end{cases}
    \]
We note that 
the second estimate of \eqref{eq:T0TpEstimate} was obtained by Marcati and Nishihara \cite{MN03}.
The first estimate of \eqref{eq:T0TpEstimate}
was obtained by Li and Zhou \cite{LZ95}.
More precisely,
they showed that $U(t) = \inf_{|x| \leq \sqrt t} \sqrt t u(x)$
blows up at a finite time
by introducing an ordinary differential inequality (ODI) with respect to $U$.
We note that on the domain $\{x; |x| \leq \sqrt t\}$,
the free solutions to \eqref{eq:DW} behaves like free heat solutions
and this property intends the ODI showing the blowup of $U$.
We also note that the coefficient $\sqrt t$ in the definition of $U$
is connected with the decay rate of free heat solutions.
We revisit their ODI argument in subsection \ref{subsection:UpperBound}.
We remark that
the argument of Li and Zhou \cite{LZ95} covers 2 dimensional case as well.
On the other hand,
in higher dimensional subcritical and critical cases,
where $p \in (1,1+2/n]$ in $n$ dimensional cases,
the second estimate of \eqref{eq:T0TpEstimate}
was obtained by
Ikeda and Wakasugi \cite{IW15} and
Ikeda and Sobajima \cite{IS19}
by using arguments for the weak form of \eqref{eq:DW},
respectively.
However with this weak form argument,
the positiveness of $M_0(u_0+u_1)$ is essential.
We also refer \cite{IO16} for related subjects.

Our interest in this paper is to consider the lifespan
when $M_0(u_0+u_1) = 0$ and $p \leq 3$.
In this case,
by using an ODE argument of \cite{FIW19},
the global non-existence is easily seen.
However, one cannot apply arguments above directly
to estimate the lifespan.
It is because in early works,
their analysis bases on a perturbation theory
and the similarity
between the asymptotic behavior of fundamental solutions of \eqref{eq:DW}
and that of free heat solutions.
When $M_0(u_0+u_1)=0$,
the asymptotic behavior of the fundamental solutions of \eqref{eq:DW}
is still similar to that of free heat equations,
so they decay faster.
This difference of decaying speed causes a problem to apply a classical approach
to investigate the lifespan for solutions to \eqref{eq:DW}.

In \cite{FG23}, a partial answer was obtained in the case where
    \begin{equation}\label{ZL2}
    u_0(x)+u_1(x) =0, \mbox{ for almost every $x$}.
    \end{equation}
Here $W^{1,p}$ denotes the usual Sobolev space of $1$st order based on $L^p$,
namely $f \in W^{1,p} \Leftrightarrow f, f' \in L^p$.
Then the following lifespan estimate holds:
\begin{Proposition}[{\cite[Theorem 1.1]{FG23}}]
\label{Proposition:T0}
Let $n=1$ and $1 < p \leq 3$.
For any $f \in (W^{1,1} \cap W^{1,p}) \backslash \{0\}$,
there exist $\varepsilon_f > 0$ and $c_f, C_f > 0$ such that
if $0 < \varepsilon < \varepsilon_f$,
$u_0 = \varepsilon f$,
and $u_1 = - \varepsilon f$,
then the corresponding mild solution
\begin{equation}\label{eq.spid4}
     u
    \in C([0,T_0); W^{1,1} \cap W^{1,p})
    \cap C^1((0,T_0); L^{1} \cap L^{p})
\end{equation}
exists uniquely with
    \begin{align}
    T_{1,p}(c_f \varepsilon^p)
    \leq T_0
    \leq T_{1,p}(C_f \varepsilon^p).
    \label{eq:T0}
    \end{align}
\end{Proposition}
\noindent
In this case, it is shown that solutions $u(t)$ to \eqref{eq:DW}
is approximated by $\partial_t S(t) u_0$ till $T_1 = c \varepsilon^{1-p}$
and
    \[
    S(t-T_1) ( u(T_1) + \partial_t u(T_1)) \sim \partial_t S(t-T_1) S(T_1) u_0
    \]
after $T_1$.
Therefore, solutions decay in the $L^1$ framework till $T_1$
and lifespan is extended as in Proposition \ref{Proposition:T0}.
We note that the second estimate of \eqref{eq:T0} is the longest possible
from the view point of \cite{FIW19}.

In the general case of
    \[
    M_0(u_0+u_1) = 0,
    \]
a similar phenomena is expected.
However, in this case,
not only Fourier $0$th moment but also higher order moment 
determine the decay rate of free solutions $S(t) f$.
It is because, Lemma \ref{Lemma:SEstimate} implies that
$S(t) f$ is approximated by $e^{t \Delta} f$
and the decay rate of $e^{t \Delta} f$ is determined by $M_k(f)$,
where $M_k(f)$ is the $k$-th moment of $f$ defined by
    \[
    M_k(f) = \int_{\mathbb R} x^k f(x) dx, 
    \]
for non-negative integer $k$.
By taking account of the decay rate of $S(t) f$,
we shall obtain the following lifespan estimate,
where $p=3/2$ is recognized as a critical exponent
in the case of $M_0(u_0+u_1) = 0$ and $M_1(u_0+u_1) \neq 0$:
\begin{Theorem}
\label{Theorem:Main}
Let\footnote{We can take initial data
    \[
    f_0\in W^{1,1} \cap W^{1,p}, f_1 \in L^1 \cap L^p
    \]
satisfying $x(f_0+f_1) \in L^1$ if $M_1(f_0+f_1) \neq 0$
and  $x^2(f_0+f_1) \in L^1$ if $M_1(f_0+f_1) = 0$.
    }
 $f_0, f_1 \in \mathcal S$  and $M_0(f_0+f_1) = 0$.
Let $\varepsilon > 0$ sufficiently small.
If $(u_0,u_1) = \varepsilon(f_0,f_1)$,
then the lifespan $T_0$ is estimated as follows:
\begin{enumerate}
    \item when $M_1(f_0+f_1) \neq 0$ and $p < 3/2$,
    \[
    c \varepsilon^{-\frac{p-1}{2-p}}
    \leq T_0
    \leq C \varepsilon^{-\frac{p-1}{2-p}}
    \]
    \item when $M_1(f_0+f_1) \neq 0$ and $p = 3/2$,
    \[
    c \varepsilon^{-2/3} e^{2 W(c\varepsilon^{-1/2})/3}
    \leq T_0
    \leq C \varepsilon^{-2/3} e^{2 W(C\varepsilon^{-1/2})/3}
    \]
    \item when $M_1(f_0+f_1) = 0$ or $p > 3/2$,
    \[
    T_{1,p}(c \varepsilon^{p})
    \leq T_0
    \leq T_{1,p}(C \varepsilon^{p}).
    \]
\end{enumerate}
\end{Theorem}
\noindent
We remark that the third case of Theorem \ref{Theorem:Main}
corresponds to Proposition \ref{Proposition:T0}.
We note that by canceling more moment of $u_0+u_1$,
the corresponding solutions may decay faster till some time but
the lifespan is not extended further
because of the viewpoint of \cite{FIW19}.

In the next section,
we show a proof of Theorem \ref{Theorem:Main}.
The key is the relation between the moments of $u_0+u_1$
and the decay rate of solutions to \eqref{eq:DW}.
In subsection \ref{subsection:LowerBound},
we show the lower bound of lifespan
by the approach of \cite{FG23}.
In order to estimate the lifespan from above,
we modify the functional of \cite{LZ95}
so as to use the first moment of $u_0+u_1$.
For details, see subsection \ref{subsection:UpperBound}.


\section{Proofs}
\subsection{Lower bound for lifespan}
\label{subsection:LowerBound}
We first note that  time-local $W^{1,1} \cap W^{1,p}$
mild solutions are  constructed
for any $u_0 \in W^{1,1} \cap W^{1,p}$ and $u_1 \in L^1 \cap L^p$
by a standard contraction argument and Lemma \ref{Lemma:SEstimate} 
with some $T>0$ and norm
    \begin{align*}
    \| u \|_{X(T)}
    &= \sup_{0 \leq t \leq T}
    \{ \|u(t)\|_{L^1}
    + (1+t)^{\frac{1}{2p'}} \|u(t)\|_{L^p} \}\\
    &+ \sup_{0 \leq t \leq T}
    (1+t)^{1/2} (\| \partial_x u(t)\|_{L^1}
    + (1+t)^{\frac{1}{2p'}} \| \partial_x u(t)\|_{L^p})\\
    &+ \sup_{0 \leq t \leq T}
   (1+t)
    (\| \partial_t u(t)\|_{L^1}
    + (1+t)^{\frac{1}{2p'}} \| \partial_t u(t)\|_{L^p}).
    \end{align*}
For detail, we refer \cite[Section 3]{MN03} and \cite[Section 2]{FG23}, for example.

Here, we recall the estimate of linear solutions without some Fourier moments.
\begin{Lemma}
\label{Lemma:SEstimateWithoutMode}
Let $f \in \mathcal S$ and $p \geq 1$.
Then the estimate
    \[
    \| S(t) f \|_{L^p}
    \leq C(t+1)^{-1/p'} \|f\|_{L^1}
    \]
holds.
Moreover,
the following estimate holds:
    \[
    \| S(t) f \|_{L^p}
    \leq C
    \begin{cases}
    (t+1)^{-1/2-1/p'} \|f\|_{L^1}
    &\mathrm{if} \quad M_0(f) = 0,\\
    (t+1)^{-1-1/p'} \|f\|_{L^1}
    &\mathrm{if} \quad M_0(f) = M_1(f) = 0.
    \end{cases}
    \]
\end{Lemma}
\noindent
Lemma \ref{Lemma:SEstimateWithoutMode} follows from Lemma \ref{Lemma:SEstimate}
and the following formal expansion of $e^{t \Delta} f$,
which follows from the Taylor expansion of heat kernel $g(x) = e^{-x^2/4}$:
    \[
    e^{t \Delta} f(x)
    = \sum_{k=0}^\infty \frac{1}{t^{k/2}} M_k(f) g^{(k)} \bigg( \frac{x}{\sqrt t} \bigg).
    \]
For the detail of relation between the decay rate of $e^{t \Delta} f$ and moments of $f$,
we refer the reader \cite{DK06,DZ92,FM01,GGS10,IKK14}, for example.

Now, we split the proof into 2 parts.

\subsubsection{case where $M_1(u_0+u_1) \neq 0$}

We put
    \begin{align*}
    \| u \|_{Y(T)}
    &= \sup_{0 \leq t \leq T}
    (t+1)^{\frac 1 2} \{ \|u(t)\|_{W^{1,1}}
    + (t+1)^{\frac{1}{2p'}} \|u(t)\|_{W^{1,p}} \}\\
    &+ \sup_{0 \leq t \leq T}
    (1+t)
    (\| \partial_t u(t)\|_{L^1}
    + (1+t)^{\frac{1}{2p'}} \| \partial_t u(t)\|_{L^p}).
    \end{align*}
We claim that mild $W^{1,1} \cap W^{1,p}$ solutions $u$ to \eqref{eq:DW} satisfies
    \begin{align}
    \|u\|_{Y(T)}
    \lesssim \varepsilon + \|u\|_{Y(T)}^p
    (T+1)^{1/2} \int_0^T (1+t)^{-\frac{2p-1}{2}} dt.
    \label{eq:ClaimY}
    \end{align}
\eqref{eq:ClaimY} implies that
if $p > 3/2$,
then for $0 \leq T \leq \widetilde T_{2,p}$,
mild solution is constructed in $Y(\widetilde T_{2,p})$
by a standard contraction argument,
where $\widetilde T_{2,p}$ satisfies
    \[
    (\widetilde T_{2,p}+1)^{1/2} \int_0^{\widetilde T_{2,p}} (1+t)^{-\frac{2p-1}{2}} dt
    = C \varepsilon^{-p+1}.
    \]
We note that $\widetilde T_{2,p}$ is given
    \[
    \widetilde T_{2,p}=
    \begin{cases}
    C \varepsilon^{-2(p-1)}
    &\mathrm{if} \quad p > 3/2,\\
    C e^{2W(C\varepsilon^{-1/2})}
    &\mathrm{if} \quad p = 3/2,\\
    C \varepsilon^{-\frac{p-1}{2-p}}
    &\mathrm{if} \quad p < 3/2,
    \end{cases}
    \]
where $W$ is the Lambert function defined by $z = W(z) e^{W(z)}$ for $z \geq 0$.
Indeed, when $p=3/2$, for $1 < T< \widetilde T_{2,p}$,
we have
    \begin{align*}
    &\sqrt{\widetilde T_{2,p}+1} \log(\widetilde T_{2,p}+1)
    = C \varepsilon^{-1/2}\\
    &\Rightarrow
    \sqrt{\widetilde T_{2,p}+1} \log \sqrt{\widetilde T_{2,p}+1}
    = C \varepsilon^{-1/2}\\
    &\Rightarrow
    \widetilde T_{2,p} +1
    = e^{2W(C \varepsilon^{-1/2})}.
    \end{align*}
Now we estimate $u(\widetilde T_{2,p})$ as
    \[
    \| u(\widetilde T_{2,p}) \|_{W^{1,1} \cap W^{1,p}}
    + \| \partial_t u(\widetilde T_{2,p}) \|_{L^{1} \cap L^{p}}
    \leq
    \begin{cases}
    C \varepsilon^p
    &\mathrm{if} \quad p > 3/2,\\
    C \varepsilon e^{-W(C \varepsilon^{-1/2})}
    &\mathrm{if} \quad p = 3/2,\\
    C \varepsilon^{\frac{3-p}{2(2-p)}}
    &\mathrm{if} \quad p < 3/2.
    \end{cases}
    \]
We note that
    \[
    \lim_{\varepsilon \to 0}
    \frac{\varepsilon^{3/2}}{\varepsilon e^{-W(C\varepsilon^{-1/2})}}
    =0.
    \]
Then by a reconstruction of solutions from $t=\widetilde T_{2,p}$
with $X(t-\widetilde T_{2,p})$ norm,
$T_0$ is estimated by
    \[
    T_0 
    \geq
    \begin{cases}
    C \varepsilon^{\frac{-2p(p-1)}{3-p}}
    &\mathrm{if} \quad p > 3/2,\\
    C \varepsilon^{-2/3} e^{2W(C \varepsilon^{-1/2})/3}
    &\mathrm{if} \quad p = 3/2,\\
    C \varepsilon^{-\frac{p-1}{2-p}}
    &\mathrm{if} \quad p < 3/2.
    \end{cases}
    \]

Now we show \eqref{eq:ClaimY}.
Lemmas \ref{Lemma:SEstimate} and \ref{Lemma:SEstimateWithoutMode} imply that
the estimate
    \[
    \| S(t) (u_0+u_1) + \partial_t S(t) u_0 \|_{L^1 \cap L^p}
    \leq C (1+t)^{-1/2}
    \]
holds.
Moreover, we have
    \begin{align*}
    \bigg\| \int_0^t S(t-\tau) |u(\tau)|^p \bigg\|_{L^1 \cap L^p}
    &\leq C \int_0^t \| u(\tau)\|_{L^p \cap L^{p^2}}^p d \tau\\
    &\leq C \| u \|_{Y(t)}^p \int_0^t (1+\tau)^{-\frac{2p-1}{2}} d\tau.
    \end{align*}
Here we note that $L^{p^2}$ is an interpolation space of $L^p$ and $W^{1,p}$.
Since derivatives of $u$ are estimated similarly,
the estimates above imply \eqref{eq:ClaimY}.
For the detail of estimates with derivatives,
we refer \cite[Proof of Theorem 1.2]{MN03}, for example.
\if{
We can have alternative and direct proof of
\begin{align}\label{eq.ebl6}
    \bigg\| \int_0^t S(t-\tau) |u(\tau)|^p \bigg\|_{L^1 \cap L^p}
    &\leq C \| u \|_{Y(t)}^p \int_0^t (1+\tau)^{-\frac{2p-1}{2}} d\tau.
    \end{align}
   Indeed,
we can proceed as follows
$$ \bigg\| \int_0^t S(t-\tau) |u(\tau)|^p \bigg\|_{ L^p} \lesssim \int_0^t (t-\tau)^{-1/(2p^\prime)} \| u(\tau)\|_{ L^{p}}^p d \tau\\
    \leq $$ $$ \leq C \| u \|_{Y(t)}^p \int_0^t (t-\tau)^{-1/(2p^\prime)} (1+\tau)^{-\frac{2p-1}{2}} d\tau$$
    On the other hand, it is not difficult to verify
    $$ \int_0^t (t-\tau)^{-1/(2p^\prime)} (1+\tau)^{-\frac{2p-1}{2}} d\tau \leq C \int_0^t  (1+\tau)^{-\frac{2p-1}{2}} d\tau, t >1,$$
    since
    $$ \int_{t-1}^t  (t-\tau)^{-1/(2p^\prime)} d\tau \sim 1, \ \int_{t-1}^t  (1+\tau)^{-\frac{2p-1}{2}} d\tau \sim t^{-\frac{2p-1}{2}} . $$
Therefore we have \eqref{eq.ebl6}.
}\fi

\subsubsection{case where $M_1(u_0+u_1) = 0$}
In this case, we put
    \begin{align*}
    \| u \|_{Z(T)}
    &= \sup_{0 \leq t \leq T}
    (1+t) \{ \|u(t)\|_{W^{1,1}}
    + (1+t)^{\frac{1}{2p'}} \|u(t)\|_{W^{1,p}} \}\\
    &+ \sup_{0 \leq t \leq T} (1+t) \{ \| \partial_t u(t)\|_{L^1}
    + (1+t)^{\frac{1}{2p'}} \| \partial_t u(t)\|_{L^p}) \}.
    \end{align*}
A similar computation implies that
for $1 < p \leq 3$,
mild solutions are constructed in $Z(C \varepsilon^{1-p})$
by a standard contraction argument.
Then we estimate
    \[
    \| u(C \varepsilon^{1-p}) \|_{W^{1,1} \cap W^{1,p}}
    + \| \partial_t u(C \varepsilon^{1-p}) \|_{L^{1} \cap L^{p}}
    \lesssim \varepsilon^p.
    \]
Therefore, the estimate
$T_0 \geq T_{1,p}(C \varepsilon^p)$.


\subsection{Upper bound for lifespan}
\label{subsection:UpperBound}
In this subsection,
we deploy a modified approach of \cite{LZ95}.
The following lemma plays an essential role.

\begin{Lemma}
\label{Lemma:ModifiedPointwiseControl}
Let $F:[0,\infty) \times \mathbb R \to [0,\infty)$.
Let $t_0 \geq 4$ and let
\[
D_\pm(t)
= \{ x \in \mathbb R; \sqrt t /2 < \pm x < \sqrt t \}.
\]
Then
    \begin{align*}
    &\inf_{x \in D_\pm (t)}
    \sqrt t \int_0^t e^{-\frac{t-\tau}{2}} \int_{|x-y|\leq t - \tau}
    I_0(\frac{\sqrt{(t-\tau)^2-(x-y)^2}}{2}) F(\tau,y) dy\\
    &\gtrsim
    \int_{t-1}^t (t-\tau) G_\pm (\tau) d \tau + \int_{t_0}^{t-1} G_\pm (\tau) d\tau,
    \end{align*}
where $G_\pm(t) = \inf_{x \in D_\pm(t)} \sqrt{t} F(t,x)$
\end{Lemma}

\begin{Remark}
In \cite{LZ95},
$D_\pm(t)$ is replaced by $\{x ,|x| < \sqrt t\}$.
Since we shall use the asymptotic form of linear solutions with first moment of initial data,
the sign restriction of $x$ is required.
\end{Remark}

\begin{proof}
For $x \in D_\pm(t)$, we obtain    
    \begin{align*}
    &\sqrt t \int_0^t e^{-\frac{t-\tau}{2}} \int_{|x-y|\leq t - \tau}
    I_0(\frac{\sqrt{(t-\tau)^2-(x-y)^2}}{2}) F(\tau,y) dy \, d\tau \\
    &\gtrsim \int_{t-1}^t G(\tau) \int\limits_{\substack{|x-y|\leq t - \tau\\ y \in D_\pm(\tau)}} dy \, d\tau
    + \int_{t_0}^{t-1} \frac{1}{\sqrt{t - \tau}} G_\pm (\tau)
    \int\limits_{\substack{|x-y|\leq \sqrt{t - \tau}\\ y \in D_\pm(\tau)}} dy \, d \tau\\
    &\gtrsim \int_{t-1}^t (t-\tau) G_\pm(\tau) d \tau
    + \int_{t_0}^{t-1} G_\pm (\tau) d \tau,
    \end{align*}
where in order to estimate the second term of the LHS of last estimate,
we have used the asymptotic behavior of $I_0$:
    \[
    I_0(y) \sim \sqrt{\frac{1}{2 \pi y}} e^y.
    \]
For the asymptotic behavior of $I_0$,
we refer \cite[8.451 5]{GR07}.
\end{proof}

Now we remind a lifespan estimate of some integral inequalities
introduced in \cite[Corollary 3.2]{LZ95}.

\begin{Lemma}[{\cite[Corollary 3.2]{LZ95}}]
\label{Lemma:C3.2LZ95}
Let $p > 1$ and $0 \leq \beta \leq 1$ and $t_0 > 4$.
Let $\varepsilon > 0$ small enough,
$T > t_0+1$, and $v: [t_0,T) \to [0,\infty)$ satisfy
    \[
    v(t)
    \gtrsim \varepsilon
    + \int_{t-1}^t (t-\tau) v(\tau)^p \tau^{-\beta} d \tau
    + \int_{t_0}^{t-1} v(\tau)^p \tau^{-\beta} d \tau.
    \]
Then
    \[
    T \leq
    \begin{cases}
    C \varepsilon^{-\frac{p-1}{1-\beta}},
    &\mathrm{if} \quad 0\leq \beta < 1,\\
    e^{-C \varepsilon^{p-1}},
    &\mathrm{if} \quad \beta = 1.
    \end{cases}
    \]
\end{Lemma}

\begin{proof}[Proof of the upper bound of Theorem \ref{Theorem:Main}]
For simplicity, we abbreviate $M_1(f_0+f_1)$ as $M_1$.
We note that when $M_1 = 0$  or $p > 3/2$,
the lifespan is estimated as \cite{FG23}.
Therefore we consider the case where $M_1 \neq 0$ and $p \leq 3/2$.
Put $w_\pm(t) = \inf_{x \in D_\pm(t)} t u(t,x)$.
When $\pm M_1 \neq 0$, 
we have
    \[
    \inf_{x \in D_\pm(t)} \{ S(t)(u_0+u_1)(x) + \partial_t S(t) u_0(x) \}
    \gtrsim t^{-1} \varepsilon | M_1|.
    \]
Therefore, Lemma \ref{Lemma:ModifiedPointwiseControl}
implies that we have,
    \begin{align*}
    w_\pm(t)
    &\gtrsim \varepsilon |M_1|
    + \sqrt t \int_{t-1}^t (t-\tau) w_\pm(\tau)^p
    \tau^{-p+\frac 1 2} d \tau
    + \sqrt t \int_{t_0}^{t-1} w_\pm(\tau)^p \tau^{-p+\frac 1 2} d \tau\\
    &\gtrsim \varepsilon
    + \varepsilon^p \sqrt t \int_{t_0}^{t-1} \tau^{-p+1/2} d \tau.
    \end{align*}
Then we have
    \[
    \inf_{t \in [t_0,\widetilde T_{2,p}]} w_\pm(t)
    \gtrsim \varepsilon.
    \]
Now we repeat the argument of \cite{LZ95} from $t=\widetilde T_{2,p}$ with $v(t) = \inf_{x \in D_\pm(t)} \sqrt t u(t)$.
Since
    \begin{align*}
    v(\widetilde T_{2,p})
    &= \widetilde T_{2,p}^{-1/2} w(\widetilde T_{2,p})\\
    &\geq
    \begin{cases}
    C \varepsilon e^{-W(C \varepsilon^{-1/2})}
    &\mathrm{if} \quad p = 3/2,\\
    C \varepsilon^{\frac{3-p}{2(2-p)}}
    &\mathrm{if} \quad p < 3/2,
    \end{cases}
    \end{align*}
Lemmas \ref{Lemma:ModifiedPointwiseControl} and \ref{Lemma:C3.2LZ95} imply that
the estimate
    \[
    T_0 \leq C v(\widetilde T_{2,p})^{-\frac{2(p-1)}{3-p}},
    \]
holds.
This coincides the first estimates of (2) and (3) in Theorem \ref{Theorem:Main}.
\end{proof}

\section*{Acknowledgement}
The first author was supported in part by
JSPS Grant-in-Aid for Early-Career Scientists No. 20K14337. The second author was partially supported by   Gruppo Nazionale per l'Analisi Matematica, by the project PRIN  2020XB3EFL with the Italian Ministry of Universities and Research, by Institute of Mathematics and Informatics, Bulgarian Academy of Sciences, by Top Global University Project, Waseda University and the Project PRA 2022 85 of University of Pisa.

\bibliographystyle{plain}
\bibliography{EL_BG}

\end{document}